\documentclass[12pt]{amsart}   
\usepackage{amsmath,amsthm,amssymb}
\usepackage{amsthm}
\usepackage[english]{babel}
\usepackage{hyperref}
\usepackage[autostyle]{csquotes}
\usepackage{csquotes}
\usepackage{IEEEtrantools}
\usepackage[dvipsnames]{xcolor}
\usepackage{enumitem, ragged2e}
\usepackage{lipsum,hyperref}
\usepackage{easyReview}
\newtheorem{theorem}{Theorem}[section]
\newtheorem{Lemma}[theorem]{Lemma}
\newtheorem{Corollary}[theorem]{Corollary}
\newtheorem{Definition}[theorem]{Definition}

\newtheorem{Example}[theorem]{Example}
\newtheoremstyle{T1.1}{}{}{}{}{\bfseries}{.}{.5em}{}
\theoremstyle{T1.1}
\newtheorem*{T1.1}{Proof of Theorem 1.1}

\newtheoremstyle{T1.2}{}{}{}{}{\bfseries}{.}{.5em}{}
\theoremstyle{T1.2}
\newtheorem*{T1.2}{Proof of Theorem 1.2}

\newtheoremstyle{Proof}{}{}{}{}{\bfseries}{.}{.5em}{}
\theoremstyle{Proof}

\newtheoremstyle{remark}{}{}{}{}{\bfseries}{\textbf{.}}{.5em}{}
\theoremstyle{remark}
\newtheorem{remark}[theorem]{Remark}

\newtheoremstyle{example}{}{}{}{}{\bfseries}{\textbf{.}}{.5em}{}
\theoremstyle{example}

\emergencystretch=\maxdimen
\begin{document} 
\title{Bi-Asymptotic $c$-Expansivity}            
	
\author[R. Nageshwar, A.G. Khan and T. Das]{Rohit Nageshwar$^1$, Abdul Gaffar Khan$^2$, Tarun Das$^{1,*}$}
	
\keywords{Decomposition theorem, Expansive, Shadowing property.\vspace*{0.08cm}\\ 
\hspace*{0.3cm} 2020 \textit{Mathematics Subject Classification.} Primary: 37B20; Secondary: 37B65.\\
\hspace*{0.17cm}
$^{*}$ Corresponding author.\\
\hspace*{0.3cm} \textit{E-mail addresses:} tarukd@gmail.com (T. Das), rohitnageshwar1705@gmail.com (R. Nageshwar) and gaffarkhan18@gmail.com (A.G. Khan). \\
\hspace*{0.17cm}$^{1}$ \textit{Department of Mathematics, Faculty of Mathematical Sciences,} \textit{University of Delhi, Delhi - 110007, India.}\\
\hspace*{0.17cm}$^{2}$ \textit{Kirori Mal College, Department of Mathematics, University of Delhi, Delhi - 110007, Delhi, India.}
	}
	
\begin{abstract}
In this paper, we define bi-asymptotically $c$-expansive maps on metric spaces and study its relationship with other variants of expansivity such as bi-asymptotically expansive maps and $N$-expansive maps. We also provide an example to establish that expansive homeomorphisms need not be bi-asymptotically expansive. Finally we prove a spectral decomposition theorem for bi-asymptotically $c$-expansive continuous surjective maps with the shadowing property on compact metric spaces. 
\end{abstract}
\maketitle
	
\section{Introduction}
In \cite{U}, Utz has introduced expansive homeomorphisms which has a significant role in the modern theory of dynamical systems. This has led to the introduction of several variants of expansivity such as positive expansivity, $c$-expansivity, bi-asymptotic expansivity and mild expansivity \cite{AH, E, LLN, LMV}.
In \cite{A,B1}, the authors have studied another important dynamical property known as shadowing property. In \cite{A1}, the author has proved that the set of all non-wandering points of expansive homeomorphisms with the shadowing property on compact metric spaces can be decomposed into finitely many  pairwise disjoint invariant subsets such that the restriction of the map on each of these subsets is topologically transitive. 
In the current literature, this result is known as spectral decomposition theorem.
This has motivated several researchers to look into the other possible decomposition of a given system in the presence of other variants of expansivity (please see \cite{B,DDS,DLRW, kDD,LLN}).
In \cite{LMV}, Morales et al. have introduced bi-asymptotically expansive maps and proved that the restriction to the set of all non-wandering points of a bi-asymptotically expansive homeomorphism with the shadowing property on a compact metric space can be decomposed into finitely many topologically transitive subsystems. 
\par
\vspace*{0.15cm}

In this paper, we introduce a weaker form of $c$-expansivity, namely bi-asymptotically $c$-expansive maps and prove such a decomposition theorem for bi-asymptotically $c$-expansive continuous surjective maps with the shadowing property on compact metric spaces. Precisely, we prove the following:

\begin{theorem}\label{T1.1}
Let $f: X \rightarrow X$ be a bi-asymptotically $c$-expansive continuous surjective map on a compact metric space $X$. If $f$ has the shadowing property, then the following holds:
		
\begin{enumerate}[leftmargin=16pt, align=left, labelwidth=\parindent, labelsep=4pt]
\item $\Omega( f )$ contains a finite sequence $B_{i}~ (1 \leq i \leq k)$ of pairwise disjoint $f$-invariant closed subsets of $X$, where $B_i$ are known as basic sets, such that
\begin{itemize}
\item[$a.$] $\Omega(f) = \bigcup\limits_{i=1}^{k}B_{i}$;
\item[$b.$] $f|_{B_{i}} : B_i \rightarrow B_i$ is topologically transitive, for each $1\leq i\leq k$.
\end{itemize}
			
\item For a basic set $B\in \lbrace B_{i}\mid 1\leq i\leq k\rbrace$, there exists an $\mathsf{m} \in \mathbb{N}^{+}$ and a finite sequence $C_{i},~ (0 \leq i \leq \mathsf{m} - 1)$ of closed subsets of $X$ such that
\begin{itemize}
\item[$a.$] $C_{i}$'s are pairwise disjoint, $f(C_{i}) = C_{i+1}$ and $f^{\mathsf{m}}(C_{i}) = C_{i}$;
\item[$b.$]  $B = \bigcup\limits_{i=0}^{\mathsf{m}-1}C_{i}$;
\item[$c.$] $f^{\mathsf{m}}|_{C_{i}} : C_{i} \rightarrow C_{i}$ is topologically mixing, for each $0\leq i\leq \mathsf{m}-1$.
\end{itemize}
\end{enumerate}
\end{theorem}

This paper is distributed as follows. In Section 2, we give necessary preliminaries required for the remaining. In section 3, we give an expansive homeomorphism which is not bi-asymptotically expansive. In section 4, we define and study the properties of bi-asymptotically $c$-expansive maps. In section 5, we give the proof of Theorem \ref{T1.1}.
	
\section{Preliminaries}
Throughout this paper, $(X,d)$ denotes a compact metric space (unless mentioned explicitly). 
For each $\epsilon > 0$ and for each $x\in X$, define  $B(x,\epsilon)=\{y\in X \mid d(x,y)<\epsilon\}$.
For a subset $B$ of $X$, the set $\overline{B}$ denotes the closure of $B$.
The set of all non-negative integers, set of all positive integers and the set of all integers are denoted by $\mathbb{N}$, $\mathbb{N}^{+}$ and $\mathbb{Z}$ respectively.
We define the distance between subsets $A,B$ of $X$ by $d(A, B)=\inf\{d(a,b) \mid a\in A\text{ and } b\in B\}$.
	
The set $X^{\mathbb{Z}} = \lbrace (x_{n}) \mid x_{n}\in X \text{, for each }n\in \mathbb{Z}\rbrace$ denotes the bi-infinite product of $X$ equipped with the product topology. 
The inverse limit space of a continuous surjective map $f:X\rightarrow X$ is a set $X_{f}= \lbrace (x_{n}) \in X^{\mathbb{Z}} \mid f(x_{n})=x_{n+1} \text{, for each }n\in \mathbb{Z}\rbrace$ which is assumed to be equipped with the subspace topology induced from $X^{\mathbb{Z}}$. 
We define a compatible metric $\mathsf{d}$ on $ X^{\mathbb{Z}}$ by 
\begin{center}
$\mathsf{d}((x_{n}),(y_{n}))=\sum\limits_{n=-\infty}^{\infty}\frac{d(x_{n},y_{n})}{2^{|n|}}$, for each pair $(x_{n}), (y_{n})\in X^{\mathbb{Z}}$.
\end{center}
We define a shift homeomorphism $\sigma : X^{\mathbb{Z}}\rightarrow X^{\mathbb{Z}}$ as $\sigma((x_{n})) = (y_{n})$, where $y_n = x_{n+1}$, for each $n\in \mathbb{Z}$ and for each $(x_{n})\in X^{\mathbb{Z}}$. Let $f:X\rightarrow X$ be a continuous surjection. We say that a subset $B$ of $X$ is $f$-invariant if $f(B)=B$. If $B\subset X$ is $f$-invariant, then $f|_{B}:B\rightarrow B$ denotes the restriction map of $f$ on $B$.
Recall that $X_{f}$ is a closed subset of $X^{\mathbb{Z}}$, $X^{\mathbb{Z}}$ and $X_{f}$ are compact spaces, $X_{f}$ is $\sigma$-invariant and $\sigma((x_{n}))=(f(x_{n}))$, for each $(x_{n})\in X_{f}$.
\par
\vspace*{0.15cm}
	
Let $f:X\rightarrow X$ be a continuous surjective map and $g:X\rightarrow X$ be a homeomorphism. The orbit of $x\in X$ under $f$ ($g$, respectively) is denoted by $\mathcal{O}^{+}_{f}(x)=\{f^{n}(x) \mid n \in \mathbb{N}\}$ ($\mathcal{O}_{f}(x) = \{g^{n}(x) \mid n \in \mathbb{Z}\}$, respectively).  
We say that $f$ ($g$, respectively) is positively expansive (expansive, respectively) with a positive expansivity (an expansivity, respectively) constant $\mathfrak{c} > 0$ if $x,y\in X$ satisfy $d(f^{n}(x),f^{n}(y))\leq \mathfrak{c}$, for each $n\in\mathbb{N}$ ($n\in \mathbb{Z}$, respectively), then $x=y$ \cite{E, U}. 
We say that $f$ is $c$-expansive with a $c$-expansivity constant $\mathfrak{c} > 0$ if $(x_n)$, $(y_n) \in X_{f}$ satisfy $d(x_{n}, y_{n})\leq \mathfrak{c}$, for each $n\in \mathbb{Z}$, then $(x_n) = (y_n)$. Recall that a homeomorphism is expansive if and only if it is $c$-expansive \cite{AH}. 
We say that $f$ is asymptotically expansive with an asymptotic expansivity constant $\mathfrak{c} > 0$ if $x,y\in X$ satisfy $d(f^{n}(x),f^{n}(y))\leq \mathfrak{c}$, for each $n\in \mathbb{N}$, then $\lim\limits_{n\rightarrow\infty} d(f^{n}(x),f^{n}(y))=0$. We say that $g$ is bi-asymptotically expansive with a bi-asymptotic expansivity constant $\mathfrak{c} > 0$ if both $g$ and $g^{-1}$ are asymptotically expansive with an asymptotic expansivity constant $\mathfrak{c}$ \cite{LMV}. 
We say that $g$ is a weak bi-asymptotically expansive homeomorphism with a weak bi-asymptotic expansivity constant $\mathfrak{c}>0$ if $x,y\in X$ satisfy $d(g^{n}(x),g^{n}(y))\leq \mathfrak{c}$, for each $n\in \mathbb{Z}$, then $\lim\limits_{n\rightarrow\pm\infty} d(g^{n}(x),g^{n}(y))=0$. We remark that this notion has been studied in \cite{LMV} without adjoining any nomenclature to it.
We say that $g$ is $N$-expansive, where $N\in \mathbb{N}^{+}$ with an expansivity constant $\mathfrak{c} > 0$ if  $\Gamma_{\mathfrak{c}}(x) =  \lbrace y \in X \mid d(g^{n}(x), g^{n}(y))\leq \mathfrak{c}, \text{ for each } n\in \mathbb{Z}\rbrace$ has at most $N$ elements,  for each $x\in  X$ \cite{M}.
\par
\vspace*{0.15cm}
	
Let $f:X\rightarrow X$ be a continuous map, $\eta > 0$ and $\xi = (x_{n})_{n\in \mathbb{N}}$ be a sequence of elements of $X$. We say that $\xi$ is an $\eta$-pseudo orbit of $f$ if $d(f(x_n), x_{n+1})\leq  \eta$, for each $n\in \mathbb{N}$, and $\xi$ is said to be $\epsilon$-traced by a point $y$ of $X$ if $d(f^{n}(y),x_{n}) \leq \epsilon$, for each $n\in \mathbb{N}$. We say that $f$ has the shadowing property if for each $\epsilon>0$, there exists a $\delta > 0$ such that each $\delta$-pseudo orbit of $f$ can be $\epsilon$-traced by some point of $X$
\cite{A,B1}.
\par
\vspace*{0.15cm}
	
Let $f:X\rightarrow X$ be a continuous map, $\alpha > 0$ and $x,y\in X$. Then $x$ is said to be a non-wandering point of $f$ if for each open set $U$ containing $x$, there exists an $n \in\mathbb{N}^{+}$ such that $f^{n}(U)\cap U\neq \emptyset$. The set of all non-wandering points of $f$ is denoted by $\Omega(f)$.
We denote $\Omega_{f} =\{(y_{n})\in \Omega(f)^{\mathbb{Z}}\mid f(y_{n})=y_{n+1}, \text{ for each } n\in \mathbb{Z}\}$.
We say that a finite sequence $x_{0}, x_{1},\ldots, x_{n}$ of elements of $X$, where  $n\in \mathbb{N}^{+}$, is a finite $\alpha$-chain if $d(f(x_{i}),x_{i+1})\leq \alpha$, for each $0\leq i < n$. We say that $x$ is $\alpha$-related to $y$ if there exists an $\alpha$-chain of the form $x=x_{0}, x_{1},\ldots, x_{n}=y$.
We say that $x$ is a chain recurrent point of $f$ if $x$ is $\alpha$-related to itself, for each $\alpha>0$. The set of all chain recurrent points of $f$ is denoted by $CR(f)$.
We say that $x$ is a periodic point of $f$ with period $n\in\mathbb{N}^{+}$ if $f^{n}(x)=x$. The set of all periodic points of $f$ is denoted by $Per(f)$. If $x$ is a periodic point of $f$ with period $1$, then $x$ is said to be a fixed point of $f$. The set of all fixed points of $f$ is denoted by $Fix(f)$. We say that a sequence $(x_{n})$ of elements of $X$ is $m$-periodic, for some $m\in \mathbb{N}^{+}$ if $x_{n}=x_{n+m}$, for each $n\in\mathbb{Z}$ \cite{AH}.
\par 
\vspace*{0.15cm}
	
Let $f:X\rightarrow X$ be a continuous map and $x\in X$. We say that $f$ is topologically transitive if for each pair of non-empty open subsets $U, V$ of $X$, there exists an integer $n> 0$ such that $f^{n}(U)\cap V \neq \emptyset$. Moreover, $f$ is said to be topologically mixing if for each pair of non-empty open subsets $U, V$ of $X$, there exists an integer $N>0$ such that $f^{n}(U)\cap V \neq \emptyset$, for each $n\geq N$. The $\omega$-limit set of $x$ is defined as 
\begin{align*}
\omega(x) =\{y\in X \mid & \text{ if there exists a sequence } ( n_{i})_{i\in \mathbb{N}} \text{ such that } \\
&\lim\limits_{i\rightarrow\infty}f^{n_{i}}(x)=y \text{, where }  n_{i}\rightarrow\infty \text{ as } i\rightarrow \infty \}.
\end{align*} 
\par
\vspace*{0.15cm}
	
Let $f:X\rightarrow X$ be a continuous surjective map, $\eta > 0$, $x\in X$ and $(x _{n})\in X_{f}$. Then,
\begin{itemize}[leftmargin=16pt, align=left, labelwidth=\parindent, labelsep=4pt]
\item  the local stable set $W_{\alpha}^{s}(x)$ is defined as
\begin{center}
$W_{\alpha}^{s}(x)=\{y\in X \mid d(f^{n}(x), f^{n}(y))\leq \alpha \text{, for each } n\in \mathbb{N}\}$.
\end{center}
\item the local unstable set $W_{\alpha}^{u}((x _{n}))$ is defined as
\begin{align*}
W_{\alpha}^{u}((x _{n}))= \lbrace y_{0} \in X  \mid & \text{ there exists a } (y_n) \in  X_{f}  \text{ such that } \\ 
& d(x_{-n},y_{-n}) \leq \alpha \text{, for each } n\in \mathbb{N}\rbrace.
\end{align*}
\item the stable set is defined as 
\begin{center}
$W^{s}(x)=\{y\in X \mid \lim\limits_{n\rightarrow \infty} d(f^{n}(x), f^{n}(y)) = 0\}$.
\end{center}
\item the unstable set is defined as
\begin{center}
$W^{u}((x _{n}))= \lbrace y_{0}\in X \mid \text{ there exists a } (y_n) \in  X_{f} \text{ such that } \lim\limits_{n\rightarrow \infty} d(x_{-n},y_{-n}) = 0\rbrace$.
\end{center}
\end{itemize}
	
\begin{remark}\label{R2.1}
Let $f:X\rightarrow X$ be a continuous surjection on a compact metric space $X$. 
\begin{enumerate}[leftmargin=16pt, align=left, labelwidth=\parindent, labelsep=4pt]
\item Then $f$ has the shadowing property if and only if for each $\epsilon >0$, there exists a $\delta>0$ such that for each $(x_{n})\in X^\mathbb{Z}$ satisfying $d(f(x_{n}),x_{n+1})<\delta$, for each $n\in \mathbb{Z}$, there exists a $(y_n)\in X_{f}$ such that $d (x_n , y_ n) < \epsilon$, for each $n\in \mathbb{Z}$ \cite[Theorem 2.3.7]{AH}.	
			
\item If $f$ has the shadowing property, then $f(\Omega(f))= \Omega(f)$ and $f|_{\Omega(f)}$ has the shadowing property \cite[Theorem 3.4.1 and Theorem 3.4.2]{AH}.
\end{enumerate}
\end{remark}
	
\section{Remarks on Bi-asymptotic Expansivity}
In \cite[Page 3]{LMV}, the authors have stated that if $f:X\rightarrow X$ is a bi-asymptotically expansive homeomorphism on a metric space $X$ (not necessarily compact), then $f$ is a weak bi-asymptotically expansive homeomorphism. The authors have further mentioned that they are not aware whether every weak bi-asymptotically expansive homeomorphism is bi-asymptotically expansive or not. In the next example, we answer this question in negative for homeomorphisms on non-compact metric spaces. Precisely, we establish the existence of a weak bi-asymptotically expansive homeomorphism on a non-compact metric space which is not bi-asymptotically expansive. 
In \cite[Proposition 11]{LMV}, the authors have also proved that every expansive homeomorphism on a compact metric space is bi-asymptotically expansive. Since the homeomorphism considered in the next example is also expansive, we get that \cite[Proposition 11]{LMV} does not hold for non-compact metric spaces. In other words, every expansive homeomorphism on a non-compact metric space need not be bi-asymptotically expansive.
	
\begin{Example}\label{E3.1}
Define $X_{0}=\{(m,e^{-m}) \mid m\in\mathbb{Z}\}$ and $X_{i}=\left\{\left(m,e^{-m+\frac{1}{i}}+\frac{1}{i}\right) \mid m\in \mathbb{Z}\right\}$, for each $i\geq 1$. Consider the space $X=\bigcup\limits_{i=0}^{\infty}X_{i}$, where $X$ is assumed to be equipped with the metric topology induced from the Euclidean metric of $\mathbb{R}^{2}$. Define a map $g:X\rightarrow X$ as follows:
		\[g(x) =\begin{cases}
\left(m+1,e^{-(m+1)+\frac{1}{i}}+\frac{1}{i}\right) & \text{ if } x = \left(m,e^{-m+\frac{1}{i}}+\frac{1}{i}\right), m\in \mathbb{Z}, i\geq 1 \\
\left(m+1,e^{-(m+1)}\right) & \text{ if } x  = (m,e^{-m}), m\in \mathbb{Z}
\end{cases}.\]
\par
Clearly, $X$ is a non-compact metric space, $g$ is a homeomorphism, $X_{i}$ are $g$-invariant sets, for each $i\in \mathbb{N}$ and $g^{-1}$ is a positively expansive homeomorphism with a positive expansivity constant $\mathfrak{c}=1$. Therefore $g$ is an expansive homeomorphism with the expansivity constant $\mathfrak{c}$. Since every expansive homeomorphism is weak bi-asymptotically expansive, we get that $g$ is a weak bi-asymptotically expansive homeomorphism also. 
Now, we claim that $g$ is not a bi-asymptotically expansive homeomorphism. Choose an $\epsilon > 0$, $k,p\in\mathbb{N}^{+}$ such that $\frac{1}{k}-\frac{1}{p}<\frac{\epsilon}{3}$ where $k<p$, $m = 1$ and an $M \in\mathbb{N}^{+}$ such that $\max \left\lbrace e^{-(1+M)+\frac{1}{k}}, e^{-(1+M)+\frac{1}{p}} \right\rbrace < \frac{\epsilon}{3}$. Clearly, for each $n\geq M$, we have
\begin{center}
$\max \left\{ d\left(g^{n}\left(\left(1,e^{-1+\frac{1}{k}}+\frac{1}{k}\right)\right),\left(1+n,\frac{1}{k}\right)\right),	
d\left(g^{n}\left(\left(1,e^{-1+\frac{1}{p}}+\frac{1}{p}\right)\right),\left(1+n,\frac{1}{p}\right)\right)\right\}$ \\
\hspace*{-3cm} $< \frac{\epsilon}{3}$.
\end{center}
Therefore we have
$d\left(g^{n}\left(\left(1,e^{-1+\frac{1}{k}}+\frac{1}{k}\right)\right),g^{n}\left(\left(1,e^{-1+\frac{1}{p}}+\frac{1}{p}\right)\right)\right) < \epsilon$, for each $n\geq M$.
In particular if we set $x=g^{M}\left(\left(1,e^{-1+\frac{1}{k}}+\frac{1}{k}\right)\right)$ and $y=g^{M}\left(\left(1,e^{-1+\frac{1}{p}}+\frac{1}{p}\right)\right)$, then we have $d(g^{n}(x),g^{n}(y)) < \epsilon$, for each $n\in \mathbb{N}$ but $\lim\limits_{n\rightarrow \infty}d(g^{n}(x),g^{n}(y)) = \left|\frac{1}{k}-\frac{1}{p}\right| >0$. Since $\epsilon > 0$ is chosen arbitrarily, we get that $g$ is not asymptotically expansive and hence $g$ is not bi-asymptotically expansive.
\end{Example}
	
\section{Bi-asymptotically $c$-expansive maps}
In this section, we first define the bi-asymptotically $c$-expansive maps and then study the behaviour of maps having such property. We further compare this notion with other existing variants of expansivity.
	
\begin{Definition}
Let $f:X\rightarrow X$ be a continuous surjective map. We say that $f$ is bi-asymptotically $c$-expansive with a bi-asymptotic $c$-expansivity constant $\mathfrak{c} > 0$ if 
\begin{enumerate}[leftmargin=16pt, align=left, labelwidth=\parindent, labelsep=4pt]
\item $(x_{n}), (y_{n})\in X_{f}$ satisfy $d(x_{n}, y_{n})\leq \mathfrak{c}$, for each $n\in \mathbb{N}$, then $\lim\limits_{n\rightarrow \infty} d(x_{n}, y_{n})=0$.
			
\item $(x_{n}), (y_{n})\in X_{f}$ satisfy $d(x_{-n}, y_{-n})\leq \mathfrak{c}$, for each $n\in \mathbb{N}$, then $\lim\limits_{n\rightarrow \infty}d(x_{-n}, y_{-n})= 0$.
\end{enumerate}
\label{D4.1}
\end{Definition}
	
\begin{remark}\label{R4.2}
Let $f:X\rightarrow X$ and $g:Y\rightarrow Y$ be continuous surjective maps on compact metric spaces $(X, d)$ and $(Y, \rho)$ respectively. Suppose that $X\times Y$ is equipped with the metric $D$ defined as $D((x_{1}, y_{1}),(x_{2}, y_{2})) = \max \lbrace d(x_{1}, x_{2}), \rho(y_{1}, y_{2}) \rbrace$, for each pair $(x_{1}, y_{1}),(x_{2}, y_{2})\in X\times Y$.
		
\begin{enumerate}[leftmargin=16pt, align=left, labelwidth=\parindent, labelsep=4pt]
\item If $f$ is bi-asymptotically $c$-expansive and $B$ is an $f$-invariant closed subset of $X$, then $f|_{B}$ is also bi-asymptotically $c$-expansive.
\item Then $f$ is bi-asymptotically $c$-expansive if and only if $f^{k}$ is bi-asymptotically $c$-expansive, for some $k \in \mathbb{N}^{+}$ if and only if $f^{k}$ is bi-asymptotically $c$-expansive, for each $k\in \mathbb{N}^{+}$.
\item If $f\times g: X\times Y\rightarrow X\times Y$ is defined as $(f\times g)(x,y) = (f(x), g(y))$, for each $(x, y)\in X\times Y$, then $f\times g$ is bi-asymptotically $c$-expansive if and only if $f$ and $g$ both are bi-asymptotically $c$-expansive.
\item If $h:X\rightarrow Y$ is a homeomorphism, then $f$ is bi-asymptotically $c$-expansive if and only if $h\circ f\circ h^{-1}$ is bi-asymptotically $c$-expansive. Therefore we get that bi-asymptotic $c$-expansivity is a dynamical property.
\end{enumerate}
\end{remark}
	
\begin{theorem}\label{T4.3}
Let $f:X\rightarrow X$ be a continuous surjective map on a compact metric space $X$. If $f$ is $c$-expansive, then $f$ is bi-asymptotically $c$-expansive. 
\end{theorem}
\begin{proof}
Let $f:X\rightarrow X$ be a $c$-expansive map on a compact metric space $X$ with a $c$-expansivity constant $\mathfrak{d}$. Choose  $0<\mathfrak{c}< \mathfrak{d}$ and recall that $W^{s}(z)=\bigcup\limits_{n\geq 0}f^{-n}\left(W^{s}_{\mathfrak{c}}\left(f^{n}(z)\right)\right)$, for each $z\in X$ \cite[Lemma 2.4.3]{AH}.
Choose $(x_{n}),(y_{n})\in X_{f} $ such that $d(x_{n},y_{n})\leq\mathfrak{c}$, for each  $n\in \mathbb{N}$. Clearly $y_{0}\in W^{s}_{\mathfrak{c}}(x_{0})$ and thus we get that $y_{0}\in W^{s}(x_{0})$. Therefore $\lim\limits_{n\rightarrow\infty}d(f^{n}(x_{0}),f^{n}(y_{0}))= 0$.
Now we suppose that $d(x_{-n},y_{-n})\leq\mathfrak{c}$, for each $n\in \mathbb{N}$. 
Clearly $d(x_{-n},y_{-n})< \mathfrak{d}$, for each  $n\in \mathbb{N} $. 
Recall that for each $\epsilon > 0$, there exists an $N\in \mathbb{N}^{+}$ such that if $d(x_{-n}, y_{-n})\leq \mathfrak{d}$, for each $n\in\mathbb{N}$, then $d(x_{-n}, y_{-n})\leq \epsilon$, for each $n\geq N$ \cite[Lemma 2.4.1]{AH}. Therefore $\lim\limits_{n\rightarrow \infty} d(x_{-n},y_{-n})= 0$. Since $(x_{n})$ and $(y_{n})$ are chosen arbitrarily, we get that $f$ is bi-asymptotically $c$-expansive with the bi-asymptotic $c$-expansivity constant $\mathfrak{c}$.
\end{proof}
	
\begin{Corollary}\label{C4.4}
Let $f:X\rightarrow X$ be a continuous surjective map on a compact metric space $X$. If $f$ is positively expansive, then $f$ is bi-asymptotically $c$-expansive. 
\end{Corollary}
\begin{proof}
Proof follows from the Theorem \ref{T4.3} and the fact that every positively expansive continuous surjective map on a compact metric space is $c$-expansive \cite[Page no. 57]{AH}.
\end{proof}
	
In the next two examples, we construct continuous surjective maps those are bi-asymptotically $c$-expansive but can not be $c$-expansive and hence can not be positively expansive also. This proves that the converse of Theorem \ref{T4.3} and Corollary \ref{C4.4} need not be true. 
	
\begin{Example}\label{E4.5}
Define a homeomorphism $f:[0,1]\rightarrow [0,1]$ by $f(x)=x^{2}$, for each $x\in [0,1]$. Recall that $f$ can not be $c$-expansive, positively expansive or expansive. Since $Fix(f) = \lbrace 0, 1\rbrace$, and $\lim\limits_{n\rightarrow \infty} f^{n}(x) = 0$, $\lim\limits_{n\rightarrow \infty} f^{-n}(x)=1$, for each $x\in (0, 1)$, we get that $f$ is bi-asymptotically $c$-expansive and bi-asymptotically expansive with  bi-asymptotic $c$-expansivity constant and  bi-asymptotic expansivity constant $0 < \mathfrak{c} < \frac{1}{2}$, respectively.
\end{Example}
	
\begin{Example}\label{E4.6}
Define $X_{0} = \bigcup\limits_{i\geq 1}\left\{ \frac{i-1}{i}, -\frac{i-1}{i},-1, 1\right\}$ and
$X_{m}=\left\{ -\frac{i}{i+1}+\frac{1}{2^{m}i(i+1)} \mid i\geq 1\right\}$, for each $m\geq 1$. 
Consider the space $X=\bigcup\limits_{m=0}^{\infty} X_{m}$, where $X$ is assumed to be equipped with the metric topology induced from the Euclidean metric of $\mathbb{R}$. 
Define a map $f:X\rightarrow X$ as follows:		
\[f(x) =\begin{cases}
x & \text{ if } x =1,-1\\
-\frac{i-2}{i-1} & \text{ if } x = -\frac{i-1}{i}, i\geq 2 \\
\frac{i}{i+1} & \text{ if } x = \frac{i-1}{i}, i\geq 1
\\-\frac{i-1}{i}+\frac{1}{2^{m}i(i-1)} & \text{ if } x= -\frac{i}{i+1}+\frac{1}{2^{m}i(i+1)},  i\geq 2, m\geq 1\\
0 & \text{ if } x  =-\frac{1}{2}+\frac{1}{2^{m+1}}, m\geq 1
\end{cases}.\]
Note that $f$ is a continuous surjective map on a compact metric space $X$, and for each $(x_{n})\in X_{f}\setminus \lbrace (\ldots, 1,1,1,\ldots), (\ldots, -1,-1,-1,\ldots) \rbrace$, we have $\lim\limits_{n\rightarrow \infty}x_{-n}=-1$  and $\lim\limits_{n\rightarrow \infty}x_{n}=1$. 
Therefore $f$ is bi-asymptotically $c$-expansive with bi-asymptotic $c$-expansivity constant $0 < \mathfrak{c} < \frac{1}{4}$. 
Now we claim that $f$ can not be $c$-expansive. Choose an $\epsilon > 0$ and a $k\in \mathbb{N}^{+}$ such that $\frac{1}{2^{k+1}} < \epsilon$. Note that if $(x_{n}),(y_{n}) \in X_{f}$ such that $x_{0}=-\frac{1}{2}, y_{0}=-\frac{1}{2}+\frac{1}{2^{k+1}}$, then $d(x_{n},y_{n}) < \epsilon$, for each $n\in \mathbb{Z}$. Since $(x_{n})\neq (y_{n})$ and $\epsilon$ is chosen arbitrarily, we get that $f$ can not be $c$-expansive.
\end{Example}
	
\begin{remark}\label{T4.7}
Let $f:X\rightarrow X$ be a homeomorphism of a compact metric space $X$. Since $(x_{n})\in X_{f}$ if and only if there exists an $x\in X$ such that $x_{n}= f^{n}(x)$, for each $n\in \mathbb{Z}$, we get that $f$ is a bi-asymptotically $c$-expansive homeomorphism with a bi-asymptotic $c$-expansivity constant $\mathfrak{c}$ if and only if $f$ is a bi-asymptotically expansive homeomorphism with the bi-asymptotic expansivity constant $\mathfrak{c}$.
\end{remark}
	
\begin{theorem}\label{T4.8}
Let $f: X\rightarrow X$ be a bi-asymptotically $c$-expansive homeomorphism on a compact metric space $X$. Then $\sigma : X_{f} \rightarrow  X_{f}$ is bi-asymptotically expansive.
\end{theorem}
\begin{proof}
Suppose that $f$ is a bi-asymptotically $c$-expansive map with a bi-asymptotic $c$-expansivity constant $\mathfrak{c}$ and $diam (X) = \sup\limits_{x, y\in X}d(x, y) = \alpha$. Choose $(x_{n}), (y_{n})\in X_{f}$ such that $\mathsf{d}(\sigma^{i}((x_{n})), \sigma^{i}((y_{n}))) \leq\mathfrak{c}$, for each $ i\in \mathbb{N}$, $\epsilon>0$ and an $N_{1}\in \mathbb{N}^{+}$ such that $\frac{\alpha}{2^{N_{1}-2}}<\epsilon$. 
Since $d(x_{i},y_{i})\leq \mathsf{d}(\sigma^{i}((x_{n})), \sigma^{i}((y_{n})))\leq\mathfrak{c}$, for each $i\in \mathbb{N}$, we get that $\lim\limits_{i\rightarrow \infty} d(x_{i},y_{i}) = 0$. Thus we can choose an $N_{2}\in \mathbb{N}^{+}$ such that $d(x_{i},y_{i})<\frac{\epsilon}{4N_{1}+2}$,  for each $i\geq N_{2}$ implying that $d(x_{i-N_{1}},y_{i-N_{1}})< \frac{\epsilon}{4N_{1}+2}$, for each $i\geq N = N_{1}+N_{2}$. Note that
\begin{align*}
\mathsf{d}(\sigma^{i}((x_{n})),\sigma^{i}((y_{n})))&=\displaystyle\sum_{n=-\infty}^{\infty}\frac{d(x_{i+n},y_{i+n})}{2^{|n|}}\\
&=\displaystyle\sum_{n=-\infty}^{-N_{1}-1}\frac{d(x_{i+n},y_{i+n})}{2^{|n|}}+\displaystyle\sum_{n=-N_{1}}^{N_{1}}\frac{d(x_{i+n},y_{i+n})}{2^{|n|}}+\displaystyle\sum_{n=N_{1}+1}^{\infty}\frac{d(x_{i+n},y_{i+n})}{2^{|n|}}\\
&< \left (\frac{\alpha}{2^{N_{1}}}+\frac{\epsilon}{2}+\frac{\alpha}{2^{N_{1}}} \right)<\epsilon \text{, for each } i>N.
\end{align*}
Hence $\lim\limits_{i\rightarrow \infty} \mathsf{d}(\sigma^{i}((x_{n})), \sigma^{i}((y_{n}))) = 0$. Similarly we can show that if $\mathsf{d}(\sigma^{-i}((x_{n})), \sigma^{-i}((y_{n})))\leq\mathfrak{c}$, for each $i\in \mathbb{N}$, then $\lim\limits_{i\rightarrow \infty} \mathsf{d}(\sigma^{-i}((x_{n})), \sigma^{-i}((y_{n})))= 0$.
Since $(x_{n})$ and $(y_{n})$ are chosen arbitrarily, we get that $\sigma$ is bi-asymptotically expansive with the bi-asymptotic expansivity constant $\mathfrak{c}$.
\end{proof}
	
\begin{Example}\label{E4.9}
Let g be an expansive homeomorphism with the shadowing property on uncountable compact metric space $(Y, d_0)$. Suppose that $g$ has infinite number of distinct periodic points $\{p_{k}\}_{k\in\mathbb{N}^{+}}$ having disjoint periodic orbits with prime period $\pi(p_{k})$, for each $k\in\mathbb{N}^{+}$.
We define $X = Y\cup E$, where $E$ is an infinite enumerable set. Assume that $Q = \bigcup\limits_{k\in\mathbb{N}^{+}}\{1, 2, \ldots, N-1\} \times \{k\} \times \{0, 1, 2, 3,\ldots, \pi(p_{k})-1\}$, and choose bijections $r : \mathbb{N}^{+} \rightarrow E$ and $s : Q \rightarrow \mathbb{N}^{+}$. Define a bijection $q : Q \rightarrow E$ by $q(i, k, j) = r(s(i, k, j))$, for each $(i, k, j) \in Q$. Note that any point $x\in E$ has the form $x = q(i, k, j)$, for some $(i, k, j)\in Q$.
		
Define a function $d : X \times X \rightarrow \mathbb{R}$ by
\begin{align*}
d(a, b) =\begin{cases}
0 & \text{ if } a = b\\
d_{0}(a, b) & \text{ if } a, b \in Y\\
\frac{1}{k} + d(g^{j}(p_{k}), b) & \text { if } a = q(i, k, j) \text{ and } b\in Y\\
\frac{1}{k} + d_0(a, g^{j}
(p_{k})) & \text { if } a \in  Y \text{ and }  b = q(i, k, j)\\
\frac{1}{k}&
\text{ if } a = q(i, k, j),b = q(l, k, j) \text{ and } i \neq l\\
\frac{1}{k}+\frac{1}{m} + d_0(g^{
j}(p_{k}), g^{r}(p_{m}))& \text{ if } a = q(i, k, j),b = q(i, m, r) \text{, } k\neq m \text{ or }  j \neq r.
\end{cases}
\end{align*}
and a map $f : X \rightarrow X$ by
\begin{align*}
f(x) =\begin{cases} 
g(x)& \text{ if } x \in Y\\
q(i, k,(j + 1)) \text{ mod } \pi(p_{k}) & \text{ if }  x = q(i, k, j)
\end{cases}.
\end{align*}
Recall that $f$ is an $N$-expansive homeomorphism with the shadowing property on a compact metric space $(X, d)$ such that $f$ is not $(N-1)$-expansive \cite{CC}. 
Since $d(f^{n}(q(1,k,0)),f^{n}(q(2,k,0)))= \frac{1}{k}$, for each $n\in \mathbb{N}$ and for each $k\in \mathbb{N}^{+}$ but $\lim\limits_{n\rightarrow \infty} d(f^{n}(q(1,k,0)),f^{n}(q(2,k,0))) \neq 0$, we conclude that $f$ can not be asymptotically expansive. From Remark \ref{T4.7}, we get that $f$ can not be bi-asymptotically $c$-expansive.
\end{Example}
	
	
\section{Proof of Theorem \ref{T1.1}}
In this section, we first prove several lemmas and then use them to prove Theorem \ref{T1.1}.
	
\begin{Lemma}\label{L5.1}
Let $f:X\rightarrow X$ be a bi-asymptotically $c$-expansive continuous surjective map on a compact metric space $X$. If $f$ has the shadowing property, then $CR(f)=\Omega(f)=\overline{Per(f)}$.
	\end{Lemma}
\begin{proof}
Recall that if $f$ is an asymptotically expansive continuous map on a compact metric space $X$ with the shadowing property, then $CR(f)=\Omega(f)=\overline{Per(f)}$ \cite[Corollary 29]{LMV}. Now the proof follows from the fact that every bi-asymptotically $c$-expansive map with a bi-asymptotic $c$-expansivity constant $\mathfrak{c}$ is also asymptotically expansive with the asymptotic expansivity constant $\mathfrak{c}$.
\end{proof}
	
\begin{Lemma}{\label{L5.2}}
Let $f:X\rightarrow X$ be a bi-asymptotically $c$-expansive continuous surjective map on a compact metric space $X$. If $f$ has the shadowing property, then there exists a $\delta >0$ such that if $p,y\in X$ with $d(p,y)<\delta$ and $y_{0}=y, p_{0}=p$, then $W^{u}((p_{n}))\cap W^{s}(y)\neq  \emptyset$ and $W^{u}((y_{n}))\cap W^{s}(p)\neq \emptyset$, for each $(p_{n}),(y_{n})\in X_{f}$.
\end{Lemma}
\begin{proof}
Let $f$ be bi-asymptotically $c$-expansive with a bi-asymptotic $c$-expansivity constant $\mathfrak{c}$.
Choose $0 < \epsilon < \mathfrak{c}$ such that if $u,v\in X$ satisfy $d(u,v) < \epsilon$, then  $d(f(u),f(v)) < \mathfrak{c}$. For this $\epsilon$, choose a $\delta>0$ by Remark \ref{R2.1}(1). Choose $p,y\in X$ such that $d(p, y) < \delta$, $p_{0}=p$, $y_{0}=y$ and $(p_{n}), (y_{n})\in X_{f}$. We define a sequence $(x_{n})\in X^{\mathbb{Z}}$ as follows:
\begin{center}
$x_{n}=\begin{cases}
p_{n}\quad &\text{if}~ n<0\\
f^{n}(y)\quad & \text{if} ~n\geq 0
\end{cases}$.
\end{center}
From Remark \ref{R2.1}(1), we can choose a sequence $(z_{n})\in X_{f}$ with $z_{0}=z$ such that $d(x_{n},z_{n})< \epsilon$, for each $n\in \mathbb{Z}$. Since $z_{n}=f^{n}(z_{0})$, for each $n\in \mathbb{N}$, we get that $d(p_{-n},z_{-n})< \mathfrak{c}$, for each $n\in \mathbb{N}$ and $d(f^{n}(y),f^{n}(z_{0}))<\mathfrak{c}$, for each $n\in \mathbb{N}$.
Therefore $\lim\limits_{n\rightarrow \infty} d(p_{-n},z_{-n}) = 0$ and $\lim\limits_{n\rightarrow \infty} d(f^{n}(y),f^{n}(z_{0})) = 0$ implying that $z_{0}\in W^{u}((p_{n}))$ and $z_{0}\in W^{s}(y)$ respectively. Thus we get that $W^{u}((p_{n}))\cap W^{s}(y)\neq \emptyset$.
Similarly we can prove that $W^{u}((y_{n}))\cap W^{s}(p)\neq \emptyset$.
\end{proof}

	
\begin{remark}\label{R5.3}
Let $f:X\rightarrow X$ be a continuous surjective map. Given $x,y\in CR(f)$, we say that $x\sim y$ if for each $\alpha >0$, there exist finite $\alpha$-chains from $x$ to $y$ and $y$ to $x$. Recall that ``$\sim$" is an equivalence relation on $CR(f)$. We write $CR(f)$ into disjoint equivalence classes $\lbrace B_{\lambda}\rbrace_{\lambda\in \Lambda}$ (called basic sets) such that each $B_{\lambda}$ is a closed subset of $X$. Moreover if $f$ has the shadowing property, then $f|_{B_{\lambda}}$ is topologically transitive with the shadowing property, for each $\lambda\in \Lambda$ \cite{AH}.
\end{remark}
	
\begin{remark}\label{R5.4}
From Lemma \ref{L5.1} and Remark \ref{R5.3}, we get that if $f:X\rightarrow X$ is a continuous surjective map on a compact metric space $X$ with the shadowing property, then $\Omega(f)=\bigcup\limits_{\lambda\in \Lambda}B_{\lambda}$, where $B_{\lambda}$ is a closed and $f$-invariant subset of $X$, for each $\lambda\in \Lambda$.
\end{remark}
	
\begin{Lemma}\label{L5.5}
Let $f:X\rightarrow X$ be a bi-asymptotically $c$-expansive continuous surjective map on a compact metric space $X$. If $f$ has the shadowing property, then $B_{\lambda}$ is open in $\Omega(f)$, for each $\lambda\in \Lambda$ and hence $\Omega(f)=\bigcup\limits_{i=1}^{k}B_{i}$, for some $k\in \mathbb{N}^{+}$.
\end{Lemma}
\begin{proof}
From Remark \ref{R2.1}(2) and Remark \ref{R4.2}(1), we get that $f|_{\Omega(f)}$ is also bi-asymptotically $c$-expansive continuous surjective map on a compact metric space $\Omega(f)$ with the shadowing property. 
Choose a $\delta>0$ by Lemma \ref{L5.2} corresponding to $f|_{\Omega(f)}$ and fix a $\lambda\in \Lambda$. We define $U_{\delta}(B_{\lambda})=\{ \mathsf{w}\in\Omega(f)\mid d(\mathsf{w},B_{\lambda})<\delta\}$. We first claim that $(U_{\delta}(B_{\lambda})\cap Per(f))\subseteq B_{\lambda}$. Choose a $p\in (U_{\delta}(B_{\lambda})\cap Per(f))$ such that $f^{m}(p)=p$, for some $m\in \mathbb{N}^{+}$, $y\in B_{\lambda}$ such that $d(p,y)<\delta$, $(y_{n})\in \Omega_{f}$ with $y_{0}=y$ and an $m$-periodic point $(p_{n})\in \Omega_{f}$ with $p_{0}=p$. From Lemma \ref{L5.2}, we get that $W^{u}((p_{n}))\cap W^{s}(y)\neq \emptyset$ and $W^{u}((y_{n}))\cap W^{s}(p)\neq \emptyset$. Since   $B_{\lambda}$ is $f$-invariant and $y\in B_{\lambda}$, we have that $y_{n}\in B_{\lambda}$, for each $n\in \mathbb{Z}$.
Now choose an $\epsilon > 0$, $q\in \omega(y)\subseteq B_{\lambda}$, $z^{1}\in W^{u}((p_{n}))\cap W^{s}(y)$ and a $z^{2}\in W^{u}((y_{n}))\cap W^{s}(p)$. Therefore there exist $(z^{1}_{n}), (z^{2}_{n})\in \Omega_{f}$ with $z^{1}_{0}=z^{1}$, $z^{2}_{0}=z^{2}$ such that $\lim\limits_{n \rightarrow\infty}d(p_{-n},z^{1}_{-n})=0$, $\lim\limits_{n \rightarrow\infty}d(f^{n}(y),f^{n}(z^{1}))=0$, $\lim\limits_{n \rightarrow\infty}d(y_{-n},z^{2}_{-n})=0$ and $\lim\limits_{n \rightarrow\infty}d(f^{n}(p),f^{n}(z^{2}))=0$. Choose an $N\in \mathbb{N}^{+}$ such that
\begin{center}
$\max \lbrace d(f(p),z^{1}_{-N+1}), d(f^{N}(y),f^{N}(z^{1})), d(y_{-N+1},z^{2}_{-N+1}), d(p,f^{N}(z^{2}))\rbrace <\epsilon$. 
\end{center} 
From the definition of $B_{\lambda}$, we can choose an $\epsilon$-chain, say $\beta$, from $q$ to $y_{-N}$. 
Since $q\in \omega(x)$, we can have $\epsilon$-chains from $p$ to $q$ and $q$ to $p$ given by
\begin{center}
$p, z^{1}_{-N+1},z^{1}_{-N+2},\ldots,z^{1}_{-1},z^{1},f^{1}(z),\ldots,f^{N-1}(z^{1}),f^{N}(y),\ldots, q$
\end{center}
and
\begin{center}
$\beta, y_{-N},z^{2}_{-N+1},\ldots, z^{2}_{-1},z^{2},f(z^{2}),\ldots, f^{N-1}(z^{2}),p$, respectively.
\end{center}
Since $\epsilon$ is chosen arbitrarily, we get that $q\sim p$ implying that $p\in B_{\lambda}$. From Lemma \ref{L5.1}, we further obtain that $U_{\delta}(B_{\lambda})= (U_{\delta}(B_{\lambda})\cap\overline{Per(f)}) \subseteq (\overline{U_{\delta}(B_{\lambda})\cap Per(f)})\subseteq B_{\lambda}$.
Hence $B_{\lambda}$ is open in $\Omega(f)$. Since $\lambda$ is chosen arbitrarily and $\Omega(f)$ is compact, we conclude that $\Omega(f)=\bigcup\limits_{i=1}^{k}B_{i}$, for some $k\in \mathbb{N}^{+}$.
\end{proof}
	
\begin{Lemma}\label{L5.6}
Let $f:X\rightarrow X$ be a bi-asymptotically $c$-expansive continuous surjective map on a compact metric space $X$. Further assume that  $\Omega(f)=\bigcup\limits_{i=1}^{k}B_{i}$ where $B_{i}$'s are basic sets, for all $1\leq i\leq k$ and for some $k\in \mathbb{N}^{+}$. If $f$ has the shadowing property and $B\in \lbrace B_{i}|1\leq i\leq k\rbrace$, then for each $p\in Per(f)\cap B$ and $C_{p}=\overline{W^{s}(p)\cap B}$ the following holds: 
\begin{enumerate}
\item $f(C_{p})= C_{f(p)}$ and $C_{p}$ is open in $B$.
\item If $q\in C_{p}\cap \text{Per}(f)$, then $C_{p}=C_{q}$.
\item If $q\in Per(f)\cap B$ such that $C_{p}\cap C_{q}\neq \emptyset$, then $C_{p}=C_{q}$.
\item There exists an $\mathsf{m}\in \mathbb{N}^{+}$ such that $B=\bigcup\limits_{i=0}^{\mathsf{m}-1}C_{f^{i}(p)}$, $f(C_{f^{i}(p)})=C_{f^{i+1}(p)}$, for each $0\leq i\leq \mathsf{m}-1$ and $f^{\mathsf{m}}|_{C_{f^{i}(p)}}:{C_{f^{i}(p)}}\rightarrow {C_{f^{i}(p)}}$ is topologically mixing, for each $0\leq i\leq \mathsf{m}-1$.
\end{enumerate}
\end{Lemma}
	
\begin{proof}
Choose a $B\in \lbrace B_{i}|1\leq i\leq k\rbrace$, $p\in Per(f)\cap B$ and least possible $m\in \mathbb{N}^{+}$ such that $f^{m}(p)=p$. Suppose that $f$ is  bi-asymptotically $c$-expansive with a bi-asymptotic c-expansivity constant $\mathfrak{c}$.
\par
\vspace*{0.15cm}

\textit{(1)} Since $B$ is $f$-invariant and $f$ is a closed map, we get that $f(\overline{W^{s}(p)\cap B})$ $=\overline {f(W^{s}(p)\cap B))}=\overline{W^{s}(f(p))\cap B}$. Hence $f(C_{p})=C_{f(p)}$.
%
%
%
%
%
We now claim that $C_{p}$ is open in $B$. From Remark \ref{R4.2}(1) and Remark \ref{R5.3}, we get that $f|_{B}$ is a bi-asymptotically $c$-expansive continuous surjective map on a compact metric space $B$ with the shadowing property. Choose a $\delta > 0$ by Lemma \ref{L5.2} corresponding to $f|_{B}$. We define $U_{\delta}(C_{p})=\{\mathsf{w}\in B\mid d(\mathsf{w},C_{p})<\delta\}$. 
We first claim that $U_{\delta}(C_{p})\cap Per(f)\subseteq C_{p}$. 
Choose a $q\in U_{\delta}(C_{p})\cap Per(f)$ with $f^{l}(q)=q$, for some $l\in\mathbb{N}^{+}$, $x\in W^{s}(p)\cap B$ such that $d(x,q)<\delta$ and an $l$-periodic orbit $(q_n )\in B_f$ with $q_{0} =q$.
From Lemma \ref{L5.2}, we get that  $W^{u}((q_n))\cap W^{s}(x)\neq \emptyset$. Thus there exists a $w\in W^{u}((q_n))\cap W^{s}(x)$ and a $(w_{n})\in B_{f}$ with $w_{0}=w$ such that $\lim\limits_{n\rightarrow \infty}d(w_{-n},q_{-n})= 0$. Since $W^{s}(x)=W^{s}(p)$ and so, $w\in W^{s}(p)$, we get that $\lim\limits_{n\rightarrow \infty}d(f^{n}(w),f^{n}(p))=0$. 
For a given $\mathsf{N}\in \mathbb{N}$, set $\mathsf{n}=ml\mathsf{N}+n$, for each $n\in \mathbb{N}$. Since $f^{ml\mathsf{N}}(w_{-ml\mathsf{N}})=w$ and $f^{ml\mathsf{N}}(p)=p$, we get that $\lim\limits_{\mathsf{n}\rightarrow\infty}d(f^{\mathsf{n}}(w_{-ml\mathsf{N}}),f^{\mathsf{n}}(p))=\lim\limits_{n\rightarrow\infty}d(f^{n}(w),f^{n}(p))=0$ implying that $w_{-ml\mathsf{N}}\in W^{s}(p)$. Since $\mathsf{N}$ is chosen arbitrarily, we get that $w_{-ml\mathsf{N}}\in W^{s}(p)\cap B$, for each $\mathsf{N}\in \mathbb{N}$ implying that $q\in \overline{W^{s}(p)\cap B}=C_{p}$. 
From Lemma \ref{L5.1}, we further obtain that $U_{\delta}(C_{p}) \subseteq	 (U_{\delta}(C_{p})\cap\overline{Per(f)}) \subseteq (\overline{U_{\delta}(C_{p})\cap Per(f)})\subseteq C_{p}$ establishes that $C_{p}$ is open in $B$.
\par
\vspace{0.15cm}

\textit{(2)} Choose a $\delta > 0$ by Lemma \ref{L5.2} corresponding to $f|_{B}$ and a $q\in C_{p}\cap Per(f)\subseteq B\cap Per(f)$ with $f^{l}(q)=q$, for some $l\in \mathbb{N}^{+}$. Recall that in \textit{(1)}, we use Lemma \ref{L5.1} to establish that $U_{\delta}(C_{q})=C_{q}$. Fix an $x\in W^{s}(p)\cap B$ such that $d(x,q)< \delta$ and so, $x\in C_{q}$. From \textit{(1)}, we get that $f^{mlj}(x)\in C_{q}$, for each $j\in \mathbb{N}$. Since $x\in W^{s}(p)$, we have $\lim\limits_{j\rightarrow \infty}d(f^{mlj}(x),p)=0$ implying that $p\in C_{q}$.
Now we choose a $y\in W^{s}(p)\cap B$. Since $y\in W^{s}(p)$, $p\in C_{q}$ and $C_{q}$ is open in $B$, we get that $\lim\limits_{j\rightarrow \infty}d(f^{mlj}(y), p)=0$ and $f^{ml\mathsf{j}}(y)\in C_{q}$, for large enough $\mathsf{j}$. We first claim that $y\in C_{q}$. 
Choose $0<\epsilon<\mathfrak{c}$. For this $\epsilon$ and $f|_{B}$, choose an $\eta>0$ from Remark \ref{R2.1}(1). Since $f^{ml\mathsf{j}}(y)\in C_{q}$, we fix a $w\in W^s(q)\cap B$ such that $d(f^{ml\mathsf{j}}(y),w)<\eta$ and define an $\eta$-pseudo orbit $(z_{n})$ in $B$ by $z_{-ml\mathsf{j}-(n+1)}\in f^{-1}(z_{-ml\mathsf{j}-n})$, $z_{-ml\mathsf{j}}=y, z_{-ml\mathsf{j}+1}=f(y),\ldots, z_{-1}=f^{ml\mathsf{j}-1}(y), z_{n}=f^{n}(w)$, for each $n\in \mathbb{N}$.
Thus there exists an $(u_{n})\in B_{f}$ such that $d(z_{n},u_{n})< \epsilon <\mathfrak{c}$, for each $n\in \mathbb{Z}$ implying that $u_{0}\in W^u((z_{n}))\cap W^{s}(w)$. Since $W^{s}(w)=W^{s}(q)$, we get that $u_{0}\in W^{s}(q)$ and $d(y,u_{-ml\mathsf{j}})=d(z_{-ml\mathsf{j}},u_{-ml\mathsf{j}})< \epsilon$.
Thus we get that $\lim\limits_{n\rightarrow\infty}d(f^{ml\mathsf{j}+n}(u_{-ml\mathsf{j}}),f^{ml\mathsf{j}+n}(q))$ $=\lim\limits_{n\rightarrow\infty}d(f^{n}(u_{0}),f^{n}(q))=0$ implying that $u_{-ml\mathsf{j}}\in W^{s}(q)\cap B$. Since $\epsilon$ and $y$ are chosen arbitrarily, we get that $W^{s}(p)\cap B\subseteq C_{q}$ and hence $C_{p}\subseteq C_{q}$.
Since $p\in Per(f)$, following similar arguments we can also establish that $C_{q}\subseteq C_{p}$ which completes the proof.
\par
\vspace{0.15cm}

\textit{(3)} Choose a $q\in Per(f)\cap B $ with $f^{l}(q)=q$, for some $l\in \mathbb{N}^{+}$ such that $C_{p}\cap C_{q}\neq \emptyset$. Since $C_{q}$ is open in $B$, we get that $W^{s}(p)\cap B \cap C_{q}\neq \emptyset$. Thus we choose an $x\in W^{s}(p)\cap B\cap C_{q}$ such that $\lim\limits_{j\rightarrow \infty}d(f^{mlj}(x), p)=0 $ and $f^{mlj}(x)\in f^{mlj}(C_{q})\subseteq C_{q}$, for each $j\in \mathbb{N}$ implying that $p\in C_{q}$. Following similar steps as in the proof of \textit{(2)}, we get that $C_{p}=C_{q}$.
\par
\vspace{0.15cm}


\textit{(4)} Since $C_{f^{m}(p)}=C_{p}$, we choose the least possible $\mathsf{m}\in \mathbb{N}^{+}$ such that $C_{f^{\mathsf{m}}(p)}=C_{p}$. From \textit{(1)} and \textit{(3)}, we get that $C_{f^{i}(p)}\cap C_{f^{j}(p)}=\emptyset$ and $f(C_{f^{i}(p)})=C_{f^{i+1}(p)}$, for all $i\neq j$ and all $0\leq i,j\leq \mathsf{m}-1$. We claim that $\bigcup\limits_{i=0}^{\mathsf{m}-1}C_{f^{i}(p)} = B$. Firstly we note that $\bigcup\limits_{i=0}^{\mathsf{m}-1}C_{f^{i}(p)} \subseteq B$. Choose a $\delta>0$ by Lemma \ref{L5.2} corresponding to $f|_{B}$, $z\in B$ and an open set $V = B(z,\delta)\cap B$ containing $z$. From \textit{(1)} and  Remark \ref{R5.3}, we recall that $C_{p}$ is open in $B$ and $f|_{B}$ is topologically transitive respectively. Thus we choose a $w\in C_{p}$ and an integer $k>0$ such that $f^{k}(w)\in V$ implying that $d(f^{k}(w), z)<\delta$. Since $f(C_{p})= C_{f(p)}$, we get that $f^{k}(w)\in C_{f^{k}(p)}$ and hence $z\in U_{\delta}(C_{f^{k}(p)})= \lbrace \mathsf{w}\in B\mid d(\mathsf{w}, C_{f^{k}(p)}) < \delta \rbrace =  C_{f^{k}(p)}\subseteq \bigcup\limits_{i=0}^{\mathsf{m}-1}C_{f^{i}(p)}$ implying that $B=\bigcup\limits_{i=0}^{\mathsf{m}-1}C_{f^{i}(p)}$. 
\par
\vspace*{0.15cm}

We now claim that $f^{\mathsf{m}}|_{C_{f^{i}(p)}}:{C_{f^{i}(p)}}\rightarrow {C_{f^{i}(p)}}$ is topologically mixing, for each $0\leq i\leq \mathsf{m}-1$. Fix $0\leq k\leq \mathsf{m}-1$. Choose a pair of non-empty open sets $U$ and $V$ of $C_{f^{k}(p)}$. From \textit{(1)} and Lemma \ref{L5.5}, we get that $C_{f^{k}(p)}$ is open in $B$ and $B$ is open in $\Omega(f)$ respectively implying that $V$ is open in $\Omega(f)$. 
We use Lemma \ref{L5.1} to choose a $q\in Per(f)\cap V$, $l\in\mathbb{N}^{+}$ such that $f^{l}(q)=q$ and an $\epsilon>0$ such that $B(q,\epsilon)\cap C_{f^{k}(p)}\subseteq V$. Since $q\in C_{f^{k}(p)}$ and sets $C_{f^{k}(p)},C_{f^{k+1}(p)},\ldots, C_{f^{\mathsf{m}-1}(p)},C_{p},C_{f(p)},\ldots,C_{f^{k-1}(p)}$ are pairwise disjoint, we get that $l=n_{k}\mathsf{m}$, for some $n_{k}\in \mathbb{N}^{+}$ and $C_{f^{k}(p)}=C_{q}$. From uniform continuity of $f^{\mathsf{m}}$, we choose an $\eta>0$ such that if $u,v\in X$ satisfy $d(u,v)<\eta$, then $d(f^{n_{k}\mathsf{m}-r\mathsf{m}}(u), f^{n_{k}\mathsf{m}-r\mathsf{m}}(v))<\epsilon$, for all $ 0\leq r\leq n_{k}-1$.
Thus we have $C_{f^{r\mathsf{m}}(q)}=f^{r\mathsf{m}}(C_{q})$ $=f^{r\mathsf{m}}(C_{f^{k}(p)})=C_{f^{k}(p)}$, for all $0\leq r\leq n_{k}-1$ implying that $W^{s}(f^{r\mathsf{m}}(q))\cap B$ is dense in $C_{f^{k}(p)}$, for all $0\leq r\leq n_{k}-1$. 
Since $U$ is open in $C_{f^{k}(p)}$, there exists a $z_{r}\in (W^{s}(f^{r\mathsf{m}}(q))\cap B)\cap U$ implying that $\lim\limits_{t\rightarrow \infty}d(f^{tn_{k}\mathsf{m}}(z_{r}),f^{tn_{k}\mathsf{m}}(f^{r\mathsf{m}}(q)))= 0$, for each $0\leq r\leq n_{k}-1$. Choose an $N_{r}\in \mathbb{N}^{+}$ such that $d(f^{tn_{k}\mathsf{m}}(z_{r}),f^{tn_{k}\mathsf{m}}(f^{r\mathsf{m}}(q)))<\eta$, for each $0\leq r\leq n_{k}-1$ and for each $t\geq N_{r}$. Therefore $d(f^{tn_{k}\mathsf{m}+n_{k}\mathsf{m}-r\mathsf{m}}(z_{r}),f^{tn_{k}\mathsf{m}+r\mathsf{m}+n_{k}\mathsf{m}-r\mathsf{m}}(q))= d(f^{tn_{k}\mathsf{m}+n_{k}\mathsf{m}-r\mathsf{m}}(z_{r}),q)<\epsilon $, for each $0\leq r\leq n_{k}-1$ and for each $t\geq N_{r}$. Since $z_{r}\in (W^{s}(f^{r\mathsf{m}}(q))\cap B)\subseteq C_{f^{r\mathsf{m}}(q)}= C_{f^{k}(p)}$ and $f(C_{p})=C_{f(p)}$, we get that $f^{\mathsf{m}(tn_{k}+n_{k}-r)}(z_{r})\in C_{f^{k}(p)}$ implying that $f^{\mathsf{m}(tn_{k}+n_{k}-r)}(z_{r})\in B(q,\epsilon)\cap C_{f^{k}(p)}\subseteq V$, for each $ 0\leq r\leq n_{k}-1$ and for each $t\geq N_{r}$.
Set $N_{k}=\max\{N_{0},N_{1},\ldots, N_{n_{k}-1}\}$ to get that $f^{\mathsf{m}(tn_{k}+n_{k}-r)}(U)\cap V\neq \emptyset$, for each $0\leq r\leq n_{k}-1$ and for each $t\geq N_{k}$. Therefore $f^{s\mathsf{m}}(U)\cap V\neq \emptyset$, for each $s> n_{k}N_{k}$. Since $U$, $V$ and $k$ are chosen arbitrarily, we get that $f^{\mathsf{m}}|_{C_{f^{i}(p)}}:{C_{f^{i}(p)}}\rightarrow {C_{f^{i}(p)}}$ is topologically mixing, for each $0\leq i\leq \mathsf{m}-1$. 
\end{proof}

\begin{T1.1}
Let $f:X\rightarrow X$ be a bi-asymptotically $c$-expansive continuous surjective map on a compact metric space $X$ with the shadowing property. Proof of the Theorem \ref{T1.1}\textit{(1)} follows from Remark \ref{R5.3}, Remark \ref{R5.4} and Lemma \ref{L5.5}, and proof of the Theorem \ref{T1.1}\textit{(2)} from Lemma \ref{L5.6}.
\end{T1.1}

\begin{flushleft}
\justifying
\textbf{Acknowledgement:} 
The first author is supported by CSIR - Junior Research Fellowship (File No.-09/045(1824)/2021-EMR-I) of Government of India.
\end{flushleft}

\end{document}